\documentclass[11pt]{amsart}

\usepackage{amsmath,amssymb,amsthm,amsfonts,enumitem,color}
\usepackage[numbers]{natbib}
\newtheorem{thm}{Theorem}[section]
\newtheorem{cor}[thm]{Corollary}

\newtheorem{prop}[thm]{Proposition}

\theoremstyle{remark}
\newtheorem{rmk}[thm]{Remark}

\newtheorem{obs}[thm]{Observation}

\theoremstyle{definition}
\newtheorem{dfn}[thm]{Definition}
\newtheorem{quest}[thm]{Question}

\newcommand{\la}{\left <}
\newcommand{\ra}{\right >}
\newcommand{\uphp}{\upharpoonright}
\newcommand{\ov}{\overline}
\newcommand{\smf}{\smallfrown}

\newcommand{\A}{\mathcal{A}}
\newcommand{\B}{\mathcal{B}}
\newcommand{\C}{\mathcal{C}}
\newcommand{\D}{\mathcal{D}}
\newcommand{\F}{\mathcal{F}}
\newcommand{\Hh}{\mathcal{H}}

\newcommand{\I}{\mathcal{I}}
\newcommand{\J}{\mathcal{J}}

\newcommand{\K}{\mathcal{K}}

\newcommand{\M}{\mathcal{M}}
\newcommand{\N}{\mathcal{N}}
\newcommand{\Oo}{\mathcal{O}}
\newcommand{\R}{\mathcal{R}}

\newcommand{\U}{\mathcal{U}}
\newcommand{\V}{\mathcal{V}}
\newcommand{\W}{\mathcal{W}}

\newcommand{\Bb}{\B_{\textrm{big}}}

\newcommand{\Wo}{\leftexp{<\omega}{\omega}}

\newcommand{\raw}{\rightarrow}

\newcommand{\ms}{\vspace{.1in}}

\newcommand{\ii}{\ov{\imath}}
\newcommand{\jj}{\ov{\jmath}}
\newcommand{\so}{\ov{s}}
\newcommand{\doh}{\ov{d}}
\newcommand{\nub}{\bar\nu}

\newcommand{\len}{<_{\textrm{len}}}
\newcommand{\lx}{<_{\textrm{lex}}}

\newcommand{\gr}{\textrm{gr}}

\newcommand{\age}{\textrm{age}}
\newcommand{\Aut}{\textrm{Aut}}
\newcommand{\qftp}{\textrm{qftp}}
\newcommand{\emtp}{\textrm{EMtp}}
\newcommand{\leftexp}[2]{{\vphantom{#2}}^{#1}{#2}}
\newcommand{\ran}{\textrm{ran~}}

\newcommand{\tr}{\textrm{stree}}
\newcommand{\str}{\textrm{strtree}}
\newcommand{\eq}{\textrm{eq}}
\newcommand{\fr}{Fra\"{i}ss\'{e}}
\newcommand{\red}{\textrm{red}}

\title{Ramsey Transfer to Semi-Retractions}
\author{Lynn Scow}
\thanks{Part of this article
is based upon work supported by the National Science Foundation
under Grant No. DMS-1928930 while the author participated in a program hosted
by the
Mathematical Sciences Research Institute in Berkeley, California, during the
Fall 2020 semester.}

\begin{document}

\begin{abstract}
We introduce the notion of a {\it semi-retraction}.  Given two structures $\A$ and $\B$, $\A$ is a semi-retraction of $\B$ if there exist quantifier-free type respecting maps $f: \B \raw \A$ and $g: \A \raw \B$ such that $f \circ g$ is an embedding.  We say that a structure has the Ramsey property if its age does.  Given two locally finite ordered structures $\A$ and $\B$, if $\A$ is a semi-retraction of $\B$ and $\B$ has the Ramsey property, then $\A$  also has the Ramsey property.  We introduce notation for what we call semi-direct product structures, after the group construction known to preserve the Ramsey property.~\cite{kpt05}  We introduce the notion of a color-homogenizing map, and use this notion to give a finitary argument that the semi-direct product structure of ordered relational structures with the Ramsey property must also have the Ramsey property.  Finally, we characterize NIP theories using a generalized indiscernible sequence indexed by a semi-direct product structure.
\end{abstract}

\maketitle

\section{introduction}

Structural Ramsey theory is the study of partition properties of classes of first-order structures.  We may consider the natural numbers as a first-order structure, $\M=(\omega,<)$, in the signature consisting of one binary relation symbol $<$ for order.    The finite substructures of $\M$ up to isomorphism form the \emph{age}, $\K$, of $\M$, which we may call $\age(\M)$ (see Section 2 for more detailed definitions).
For every integer $n \geq 1$, there is a unique linear order $\A_n \in \K$.  Let ${ \M \choose \A_n}$ denote all substructures of $\M$ isomorphic to $\A_n$.  Given an integer $k \geq 1$, a \emph{$k$-coloring} of ${\M \choose \A_n}$ is a function $f: {\M \choose \A_n} \raw k$.  It is clear that $\{f^{-1}(i) \mid i \in k\}$ forms a finite partition of ${\M \choose \A_n}$.
Ramsey's theorem for finite sequences states that for any integers $k, n, m \geq 1$, there exists an integer $N$ such that for any $k$-coloring $f$ of ${\A_N \choose \A_n}$ there exists $\B \subseteq \A_N$ such that $\B \cong \A_m$ and $f \uphp {\B \choose \A_n}$ is a constant function.~\cite{ra29}
The property just described is called the \emph{Ramsey property (RP) for $\K$}, and may be stated for any class $\K$ of finite structures in some signature.
We say that a locally finite structure $\M$ has RP if $\age(\M)$ has RP.  
In this paper, we consider RP only for structures $\M$ that are locally finite and ordered by some $0$-definable relation.
To read a survey of some recent work in structural Ramsey theory, please see \cite{nvt}.

We consider the following:
\begin{quest}\label{qu}
What mechanisms transfer RP from one ordered structure to another?  
\end{quest}
\noindent It is natural to ask whether taking reducts or expansions could preserve RP, perhaps under some additional assumptions.
In \cite{ne05}, it is shown that the age of any linearly ordered structure with RP must have the amalgamation property (AP). 
Thus, if $\B$ is an ordered structure with RP and $\A$ is an ordered reduct of $\B$ that fails to have AP, then $\A$ fails to have RP.
For example, let $\I_0 = (\Wo, \unlhd, \wedge, \lx)$ where $\unlhd, \wedge, \lx$
are defined as in Definition \ref{trees}.
$\I_0$ is shown to have RP in \cite{le73} but the reduct $\I_t = \I_0 \uphp \{ \unlhd,\lx\}$, though ordered, has an age that fails to have AP, and thus fails to have RP (see \cite{sok151} or Corollary 3.19 of \cite{sc15} for a discussion).  

There are many examples of ages of ordered structures that have AP but not RP, and the class of all finite partial orders with an added linear order is one such age (see Lemma 4 in \cite{sok111}).  The class of all finite equivalence relations that are linearly ordered is another example of an ordered class with AP but not RP, though the class of all finite equivalence relations with a convex linear order does have RP. \cite{kpt05}  

Furthermore, given an age with RP, not all ordered reducts with AP have RP.
Let $\K_1$ be the class of all finite convexly ordered equivalence relations in the signature $\sigma_1=\{E_1,<_1\}$, and let $\K_2$ be the same in a disjoint signature $\sigma_2=\{E_2,<_2\}$.  As in Definition 3.21 of \cite{bod15}, define the \emph{free superposition} $\K_1 \ast \K_2$, which is the class of all finite $(\sigma_1 \cup \sigma_2)$-structures whose $\sigma_i$-reduct is in $\K_i$.  The \fr~limit of $\K_1 \ast \K_2$ has RP, by Theorem 3.24 in \cite{bod15}, but the ordered reduct to $\{E_1,E_2,<_1\}$ does not, even though its age has AP.
It is worth noting that, in some cases, one may start with a class of finite structures that is not ordered and does not have AP or RP, and achieve these properties by expanding the signature.  
The case of bowtie-free graphs in \cite{hune18} is one such example.
It is also natural to consider interpretations of one structure in another in relation to Question \ref{qu} (see \citep[Definitions 7.1, 7.6 in \textit{Models and Groups}]{kama94} for background on interpretations).
Recent work has shown how notions of interpretability of one structure in another may transfer RP: ``simply bi-definable'' expansions in \citep[Proposition 9.1]{kpt05} and Ramsey expansions of a structure interpretable in a Ramsey structure in \citep[Proposition 3.8]{bod15}.
By a well-known result from \cite{kpt05}, a closed subgroup $G \leq S_{\infty}$ is extremely amenable if and only if $G$ is the automorphism group of an ordered \fr~limit with RP. 
Countable structures $\A$ and $\B$ have homeomorphic automorphism groups if and only if $\A$ and $\B$ are infinitarily bi-interpretable (see \citep[Corollary 7.7 in \textit{Models and Groups}]{kama94} for a proof).
By a combination of these results, given two  linearly ordered \fr~limits $\F_1$ and $\F_2$, if $\F_1$ and $\F_2$ are infinitarily bi-interpretable, then $\F_1$ has RP if and only if $\F_2$ has RP.

In this paper, we introduce the notion of a \textit{semi-retraction} (see Definition \ref{sr}).  A semi-retraction has some elements in common with an infinitary interpretation.  
In Corollary \ref{transfer}, we show that for any locally finite ordered structures $\A$ and $\B$, if $\B$ has  RP and $\A$ is a semi-retraction of $\B$, then $\A$ also has  RP.  
In Section 2, we explain our notation as well as some background on RP, the modeling property and generalized indiscernible sequences.  In Section 4, we define {color-homogenizing maps} in Definition \ref{ch} and prove the corresponding RP transfer result in Theorem \ref{thm2}.  In Section 5, we define semi-direct product structures in Definition \ref{sdps} and apply Theorem \ref{thm2} to obtain a finitary argument that the semi-direct product structure obtained from ordered relational structures with RP has RP (Theorem \ref{thm3}).  In Section 6, we deduce examples of structures with RP as special cases of Theorem \ref{thm3}, some of which are known.   We also prove a characterization of NIP theories using a generalized indiscernible sequence indexed by the semi-direct product structure named in Corollary \ref{examples}(2)(see \cite{scgh17} for more results of this kind).

\section{preliminaries}

\subsection{Notation and conventions}

The notation $\raw$ is reserved for the function arrow in $f: A \raw B$ as well as for the Erd\H{o}s-Rado partition arrow in $\C \raw (\B)^\A_k$ (see \ref{notions} \textit{Ramsey notions}).  The notation $\Rightarrow$ is reserved for the material conditional.

\subsubsection{Size and order}
For a set $X$, $|X|$ is the cardinality of $X$.  
We follow the logical convention that $\omega$ is the set of non-negative integers, and for any $n \in \omega$, $n= \{0, 1, \ldots, n-1\}$.  Thus, ``$i<n$'' and ``$i \in n$'' may be used interchangeably, for any $i \in \omega$.
The notation $<$ is reserved for the linear order on $\omega$, and a different symbol is used for linear orders on other structures.
An \emph{order} is a linear order unless otherwise specified.
For sets $A, B$, by $\leftexp{A}{B}$, we mean the set of all functions $f: A \raw B$.
Tuples $\ov{a}$ from $\A$ are finite sequences $(a_i)_{i < n}$ for some $n \in \omega$ and for some $a_i \in \A$, for all $i < n$.
We define $(\ov{a})_i=a_i$ and $\ran \ov{a} = \{a_i \mid i < n\}$.
We reserve the notation $\ell$ for the function that outputs the length of a tuple: e.g. $\ell(\ov{a}) = n$, when $\ov{a} = (a_i)_{i<n}$.
Given a function $f: A \raw B$ and a tuple $\ov{a} = (a_0,a_1,\ldots,a_{n-1})$ from $A$, $f(\ov{a})$ is defined to be the tuple $(f(a_0),f(a_1),\ldots,f(a_{n-1}))$ from $B$.
Given some $m$-tuples $\ov{a}_{j_i}$ for all $i < n$, by $\ov{a}_{\ov{\jmath}}$ we mean the $m \cdot n$-tuple $(\ov{a}_{j_0}, \ldots, \ov{a}_{j_{n-1}})$.

\subsubsection{Structures} 

A \emph{signature} is a set of relation and function symbols with 
assigned arities (where $0$-ary function symbols play the role of constant symbols).
Given a signature $L$, an \emph{$L$-structure} $\A$ consists of an underlying set $|\A|$ with interpretations of all symbols in $L$ as relations or functions on $|\A|$ of the correct arity. 
For example, in the case that $R \in L$ is a relation symbol of arity $n$, $\A$ interprets $R$ as some relation $R^\A \subseteq |\A|^n$.  (See \cite{ho93} as a reference for common model-theoretic terms.) 
As usual, $a \in \A$ means that $a \in |\A|$, and we may use the symbol $R$ to stand for its interpretation, if the intended structure $\A$ is understood.
We say that $\A$ is a structure \emph{on $|\A|$ in the signature $L$}.  
The cardinality of a structure, $\A$, is denoted by $||\A||$.
The \emph{(first-order) language} of $L$ is the set of all first-order $L$-formulas.
Given a structure $\A$, we use $\sigma(\A)$ to refer to the signature of $\A$ and $L(\A)$ to refer to the language of $\A$.
By an \emph{ordered structure} $\A$, we mean one that is linearly ordered by 
some binary relation symbol in $\sigma(\A)$.
A subset $X \subseteq |\B|^n$ is \emph{0-definable (in $\B$)} if there exists an $n$-ary formula $\varphi(v_0,\ldots,v_{n-1})$ in $L(\B)$ such that for every $\ov{b} \in |\B|^n$, $\ov{b} \in X$ if and only if $\B \vDash \varphi(\ov{b})$.
Given two $L$-structures $\A_1, \A_2$ and $L' \subseteq L$, an \emph{$L'$-embedding}
$\sigma: \A_1 \raw \A_2$ is an injection from $|\A_1|$ into $|\A_2|$ such that for any $L'$-formula $\varphi$, for all $\ov{a}$ from $\A_1$, $\A_1 \vDash \varphi(\ov{a}) \Leftrightarrow \A_2 \vDash \varphi(f(\ov{a}))$.  
Given two $L$-structures $\A_1, \A_2$, an \emph{embedding} $\sigma: \A_1 \raw \A_2$ is assumed to be an $L$-embedding.
Given two $L$-structures $\A_1, \A_2$, an \emph{($L$-)isomorphism} $\sigma : \A_1 \raw \A_2$ is an embedding that is surjective onto $|\A_2|$, and thus $\sigma^{-1}: \A_2 \raw \A_1$ is also an embedding.
We denote that $\A_1, \A_2$ are isomorphic by
$\A_1 \cong \A_2$, or $\A_1 \cong_L \A_2$, for clarity.
For two $L$-structures $\A, \B$, $\A \subseteq \B$ means that $\A$ is a \emph{substructure} of $\B$, i.e., the identity map is an $L$-embedding from $\A$ into $\B$.

Given a relational structure $\A$ and a 0-definable subset $\D \subseteq |\B|$ suppose that there exists a bijection $f: \A \raw \D$ such that for every $n$ and 
$n$-ary relation symbol $R \in \sigma(\A)$,
there is a set $\hat{R} \subseteq |\B|^n$ that is 0-definable in $\B$ such that $\ov{a} \in R^\A$ if and only if $f(\ov{a}) \in \hat{R}$.  In this case we say that $\A$ is a \emph{reduct} of $\B$ (see \cite{ev94} for a statement of the more general case).
Given an $L$-structure $\B$ and a subset $L' \subset L$, by $\B \uphp L'$ we mean the $L'$-structure on $|\B|$ obtained from $\B$ by restricting to the symbols in $L'$.  
We refer to $\B \uphp L'$ as the \emph{$L'$-reduct} of $\B$.

The \emph{age}, $\K$, of a structure $\I$, denoted by $\K = \age(\I)$, is the collection of all finitely-generated substructures of $\I$, up to isomorphism.  (In the case that the signature of $\I$ is relational, $\age(\I)$ is the collection of all finite substructures of $\I$, up to isomorphism.)  Every age has the hereditary property and the joint embedding property (JEP)  (see \cite{ho93} for a reference).
We say that an age $\K$ has the \emph{amalgamation property (AP)} if given any structures $\A, \B, \C \in \K$ and embeddings $f_1 : \A \raw \B, f_2 : \A \raw \C$ (what we call an \emph{amalgamation problem in $\K$}), there exist a structure $\D \in \K$ and embeddings $g_1 : \B \raw \D, g_2: \C \raw \D$ (what we call a \emph{solution to the amalgamation problem}) such that $g_1 \circ f_1 = g_2 \circ f_2$.  
A structure $\M$ is \emph{ultrahomogeneous} if any isomorphism between finitely-generated substructures of $\M$ extends to an automorphism of $\M$.
The age of a countable, ultrahomogeneous structure in a countable signature has AP, and for every nonempty countable age $\K$ with AP in a countable signature, there is a countable ultrahomogeneous structure with age $\K$, $\M$, (unique, up to isomorphism) which we refer to as the \fr~limit of $\K$, \emph{Flim $\K$}.~\cite{fr54}  

Given any signature $L = \{R,\prec\}$ that consists of two binary relation symbols, the \emph{random ordered graph in the signature $L$} is defined to be any structure isomorphic to Flim $\K$, where $\K$ is the class of all finite $L$-structures (up to isomorphism) that are linearly ordered by $\prec$ and that interpret $R$ as a graph edge relation (symmetric, with no loops).  The relation that plays the role of $\prec$ should be clear from context.

\subsubsection{Types}
%
Given an integer $n \geq 1$ and an $n$-tuple $\ov{a}$ from $\A$, the \emph{quantifier-free type of $\ov{a}$ in $\A$}, which is denoted by $\qftp^\A(\ov{a})$, is the set of all $n$-ary quantifier-free formulas in $L(\A)$ satisfied by $\ov{a}$ in $\A$.
It is typical to write some subtype for $\qftp^\A(\ov{a})$, whose closure under logical consequence in $\A$ is $\qftp^\A(\ov{a})$.
We say that $\eta$ is a \textit{complete quantifier-free $n$-type} in $\A$ if $\eta = \qftp^\A(\ov{a})$ for some length-$n$ tuple $\ov{a}$ from $\A$.  
We say that $\eta$ is a \emph{complete quantifier-free type} if it is a complete quantifier-free $n$-type for some $n$.
We say that $\A \vDash \eta(\ov{a})$ if and only if $\A \vDash \theta(\ov{a})$, for all $\theta \in \eta$.
The notation $\eta(\ov{a})$ (similarly, $\theta(\ov{a})$) presupposes that $\ov{a}$ is a tuple of the correct length, and we may sometimes specify the length of $\ov{a}$ to aid in clarity.
We define $\eta(\A)$ to be all tuples 
from $\A$ that satisfy $\eta$ in $\A$.
 Given a structure $\A$, we write $ \ov{a} \equiv_\A \ov{b}$ to mean that $\A \vDash \varphi(\ov{a}) \Leftrightarrow \A \vDash \varphi(\ov{b})$, for all formulas $\varphi \in L(\A)$.
 We write $\ov{a} \sim_\A \ov{b}$ to mean that $\A \vDash \theta[\ov{a}] \Leftrightarrow \A \vDash \theta[\ov{b}]$, for all quantifier-free formulas $\theta$ in $L(\A)$.
The statement
$\ov{a} \sim_\A \ov{b}$
is equivalent to the statement 
$\qftp^\A(\ov{a}) = \qftp^\A(\ov{b})$
which is equivalent to the statement that
the map $a_i \mapsto b_i$ extends to an isomorphism of the structures generated by $\ov{a}$ and $\ov{b}$.

\subsubsection{Ramsey notions}\label{notions}
For any integer $k \geq 1$, a \emph{$k$-coloring} of a set $X$ is any function $c: X \raw k$.  
A \emph{copy of $\A$ in $\B$} is a substructure $\A' \subseteq \B$ where $\A' \cong \A$.  The set of all copies of $\A$ in $\B$ is denoted by ${\B \choose \A}$.  
Assuming that $\A$ is a structure ordered by the relation $\prec$, for any tuple $\ov{b}$ from $\A$, we say that $\ov{b}=(b_i)_{i<n}$ is an \emph{increasing tuple} if $b_i \prec b_j \Leftrightarrow i<j$, for all $i,j<n$.
Moreover, the \emph{increasing enumeration} of $\A$ is the increasing tuple $\ov{a}$ such that $\ran \ov{a} = |\A|$.
By an \emph{increasing copy of $\A$ in $\B$} we mean the increasing enumeration of $\A'$, where $\A'$ is some copy of $\A$ in $\B$.
We work with the following definition of the Ramsey property (see \cite{ne05,kpt05}).
\begin{dfn}\label{rpDef} We say that an age, $\K$, of finite structures has the $\A$-\emph{Ramsey property} if for all $\B \in \K$ and for any integer $k \geq 2$, there exists $\C \in \K$ such that for any $k$-coloring $c$ of ${\C \choose \A}$, there is a structure $\B'$ $\in {\C \choose \B}$ such that for any $\A', \A'' \in {\B' \choose \A}$, $c(\A') = c(\A'')$.  

We say that $\C$ is \emph{Ramsey} for $\A, \B, k$ and denote this property of $\C$ by the expression:
$$\C \raw (\B)^\A_k$$

We say that $\B'$ is a copy of $\B$ that is \emph{homogeneous for $c$ (on copies of $\A$)}.

We say that $\K$ has the \emph{Ramsey Property (RP)} if it has the $\A$-Ramsey property for all $\A \in \K$.

We say that a locally finite structure $\I$ has RP if age($\I$) has RP.
\end{dfn}

\begin{obs}\label{equiv}  Note that we obtain an equivalent definition of RP if we replace the arbitrary $k$-coloring $c: {\C \choose \A} \raw k$ in the definition with any function $c': {\C \choose \A} \raw Y$, where $Y$ is any set of cardinality $k$.
\end{obs}

We give a slight rephrasing of Theorem 4.2(i) from \cite{ne05} as Theorem \ref{nesAP}:

\begin{thm}\label{nesAP} If $\K$ is an age of ordered structures and $\K$ has RP, then $\K$ has AP.
\end{thm}

Theorem \ref{rgRP} is a well-known result from  \cite{abha78,rone77}: 

\begin{thm}\label{rgRP} The age of any random ordered graph has RP.
\end{thm}

\begin{rmk}\label{finite} If a finite ordered structure $\I$ has RP, then there are no isomorphisms between distinct subsets of $\I$.  This is easy to show letting $\I$ play the role of $\B$ in Definition \ref{rpDef}.  
\end{rmk}

\subsection{The modeling property}

In the study of classification theory in model theory there has been significant use of generalized indiscernible sequences, named ``$\I$-indexed indiscernible sets'' in \cite{sh90}.

\begin{dfn} Fix a structure $\I$, an integer $l \geq 1$, and $l$-tuples $\ov{a}_i$ from some structure $\M$, for all $i \in \I$.  We say that $(\ov{a}_i \mid i \in \I)$ is an \textbf{$\I$-indexed indiscernible set} if for any integer $n \geq 1$, for all $n$-tuples $\ii,\jj$ from $\I$,
$$\ii \sim_\I \jj \Rightarrow  \ov{a}_{\ii} \equiv_\M \ov{a}_{\jj} .$$

We say that $(\ov{a}_i \mid i \in \I)$ is an \textbf{$\I$-indexed indiscernible sequence} if $\I$ is an ordered structure, or a \textbf{generalized indiscernible sequence} if $\I$ is an ordered structure that is clear from context.
\end{dfn}

We repeat definitions from \cite{sc15} as Definition \ref{em} and Definition \ref{mpd}.

\begin{dfn}\label{em}  Given an integer $l \geq 1$, an $L'$-structure $\I$, an $L$-structure $\M$ and an $\I$-indexed set of $l$-tuples from $\M$, $X = (\ov{a}_i \mid i \in \I)$, we define the \emph{EM-type of $X$ ($\emtp(X)$)} to be a syntactic type in variables $(\ov{x}_i \mid i \in \I)$, where $\ell(\ov{x}_i) = l$ for each $i \in \I$, as follows:
$$\emtp(X) = \{ \psi(\ov{x}_{i_0},\ldots,\ov{x}_{i_{n-1}}) \mid \psi \in L, \ov{\imath} \in \leftexp{n}{\I} \textrm{~and~} (\forall  \ov{\jmath} \in \leftexp{n}{\I})( \ov{\jmath} \sim_\I \ov{\imath} \Rightarrow \M \vDash \psi(\ov{a}_{j_0},\ldots,\ov{a}_{j_{n-1}})) \}$$
\end{dfn}

Proposition \ref{useful} is a useful equivalence which follows directly from Definition \ref{em} (see Proposition 2 of \cite{sc15} for more details):

\begin{prop}\label{useful}Given an $L'$-structure $\I$ and an $L$-structure $\M$,
fix sets of $l$-tuples from $\M$ indexed by $\I$,
$X = (\ov{a}_i \mid i \in \I)$ and
$Y = (\ov{b}_i \mid i \in \I)$.
$Y \vDash \emtp(X)$ if and only if for any integer $n \geq 1$, for all complete quantifier-free $n$-types $\eta$ in $\I$ and all $n \cdot l$-ary formulas $\varphi \in L$, if
$$(\forall \jj) ( \I \vDash \eta(\jj) \Rightarrow \M \vDash \varphi(\ov{a}_{\jj}))$$
then
$$(\forall \jj) ( \I \vDash \eta(\jj) \Rightarrow \M \vDash \varphi(\ov{b}_{\jj}))$$
\end{prop}

If $\I$ is ordered by a 0-definable relation in $\I$, it is trivial to produce $\I$-indexed indiscernible sets, by Ramsey's theorem for finite sequences.  The following property guarantees that we can produce $\I$-indexed indiscernible sets that witness additional structure.

\begin{dfn}\label{mpd} 
Given a structure $\I$, we say that $\I$-indexed indiscernible sets have the \textbf{modeling property} if for any  integer $l \geq 1$, any $|\I|^+$-saturated structure $\M$, and any $\I$-indexed set of $l$-tuples from $\M$ 
$$X = (\ov{a}_i \mid i \in \I) ,$$
there exists an $\I$-indexed indiscernible set of $l$-tuples from $\M$
$$Y = (\ov{b}_i \mid i \in \I)$$
such that $Y \vDash \emtp(X)$.

We say that $Y$ is \textbf{locally based} on $X$.
\end{dfn}

\begin{thm}\label{character} Suppose that $\I$ is a locally finite ordered structure.
 $\I$-indexed indiscernible sequences have the modeling property if and only if $\age(\I)$ has RP.
\end{thm}

\begin{proof} In Theorem 3.12 of  \cite{sc15}, this result is stated for locally finite ordered structures $\I$ with the additional condition \texttt{qfi}, which stands for ``quantifier-free types are isolated by quantifier-free formulas".  Theorem 3.12 in \cite{sc15} generalizes Theorem 4.31 in \cite{sc12} which is stated only for ordered  structures $\I$ in a finite relational signature.  In fact, it was later pointed out to the author that the \texttt{qfi} assumption is not needed (see the \textit{Acknowledgements} section).  To see this, in the argument for \citep[Claim 3.13]{sc15}, replace $L'$ with an expansion $L''$ such that $L'' \setminus L'$ consists of a predicate $p_\A(\ov{x})$ for the quantifier-free type of the increasing enumeration of $\A$, for every finite substructure $\A \subseteq \mathcal{\I}$.  Then apply compactness to the type $S$ where we replace $T_\forall \cup \textrm{Diag}(\mathcal{\I})$ with the diagram of $\mathcal{\I}$ in $L''$.  It is noted in the proof of \citep[Theorem 3.12]{sc15} that the \texttt{qfi} hypothesis is used only in the argument for Claim 3.13.
By the present argument, we see why it is not even needed there.  
\end{proof}

\section{transfer by semi-retractions}

For concrete examples, we give the definitions of the \textit{Shelah tree} $\I_\tr$, the \textit{strong tree} $\I_\str$ and the \textit{convexly ordered equivalence relation} $\I_\eq$.  All three structures are locally finite ordered structures.  An exposition of the proof that $\I_\tr$- and $\I_\str$-indexed indiscernible sequences have the modeling property is given in \cite{kks14}.  A proof that $\I_\eq$ has RP is given in \citep[Theorem 6.6]{kpt05}.  This latter fact is equivalent to array-indiscernible sequences having the modeling property, and array-indiscernible sequences have been a common tool in model theory (\citep[Lemma 5.6]{kks14} provides a direct proof of the modeling property).

\begin{dfn}\label{trees}
\begin{itemize}
\item Define $\I_\tr$ to be the structure on $\Wo$ (finite sequences from $\omega$) in the signature
 $\{\unlhd, \wedge, \lx, \{P_n\}_{n \in \omega} \}$ where for all $\eta, \nu \in \Wo$, $\eta \unlhd \nu$ if and only if $\eta$ is an initial segment of $\nu$, $\wedge$ is the meet in the partial order $\unlhd$, $\lx$ is the lexicographic order on finite sequences, i.e. $\eta \lx \nu$ if and only if
$$\eta \unlhd \nu \vee \eta(  \ell(\eta \wedge \nu) ) < \nu( \ell(\eta \wedge \nu) ) ,$$

\noindent and $ \eta \in P_n \Leftrightarrow \ell(\eta) = n$, for all $n \in \omega$.
\item Define $\I_\str$ to be the structure on $\Wo$ in the signature $\{\unlhd, \wedge, \lx, \len \}$ where $\unlhd, \wedge, \lx$ are interpreted as in $\I_\tr$ and $\len$ is the preorder on $\mu, \nu \in \Wo$ defined by the lengths of the sequences:
$$\mu \len \nu \Leftrightarrow \ell(\mu) < \ell(\nu)$$
\item Define $\I_\eq$ to be the structure on $\omega \times \omega$ in the signature $\{ E, \prec \}$ where for all $(i,j),(s,t) \in \omega \times \omega$, $(i,j)E(s,t) \Leftrightarrow i=s$ and 
$(i,j) \prec (s,t) \Leftrightarrow i<s \vee (i=s \wedge j<t)$.
\end{itemize}
\end{dfn}

\begin{dfn}  Given any structures $\A, \B$, we say that an injection $h: \A \raw \B$ is \textbf{quantifier-free type respecting (qftp-respecting)} if for all finite same-length tuples $\ii, \jj$ from $\A$,
$$\ii \sim_\A \jj \Rightarrow h(\ii) \sim_\B h(\jj) .$$
\end{dfn}

In the following Definition, the term ``semi-retraction'' is inspired by the definition of ``retraction'' from \cite{ahzi86}.  
Corollary 1.4 in \cite{ahzi86} gives an equivalent condition for a countable, $\aleph_0$-categorical structure in a countable signature to be a retraction (attributed to T. Coquand):
$\A$ is a retraction of $\B$ if and only if there are continuous homomorphisms
$$\Aut(\A) \overset{\varphi}\raw \Aut(\B) \overset{\psi}\raw \Aut(\A)$$
such that $\psi \circ  \varphi = 1$.  

\begin{dfn}[semi-retractions]\label{sr} Let $\A$ and $\B$ be any structures.  We say that $\A$ is a \textbf{semi-retraction of $\B$} if
there exist qftp-respecting injections $g: \A \raw \B$ and $f: \B \raw \A$ such that for any complete quantifier-free type $\eta$ in $\A$, for any $\ov{s}$ from $\A$,
\begin{enumerate}[label=(\roman*)]
\item\label{red3} $\A \vDash \eta(\ov{s}) \Rightarrow \A \vDash \eta((f \circ g)(\ov{s}))$
\end{enumerate}
\end{dfn}

\begin{obs}\label{csb} If $\A$ is a semi-retraction of $\B$, then $||\A||=||\B||$, by the Schr\"{o}der-Bernstein theorem.
\end{obs}

\begin{thm}\label{thm1}  Let $\A$ and $\B$ be any 
structures.  Suppose that $\A$ is a semi-retraction of $\B$.  Furthermore, suppose that $\B$-indexed indiscernible sets have the modeling property.  
Then $\A$-indexed indiscernible sets have the modeling property.
\end{thm}

\begin{proof} Fix structures $\A$ and $\B$ such that $\A$ is a semi-retraction of $\B$ and assume that $\B$-indexed indiscernible sets have the modeling property.  Fix an integer $l \geq 1$ and an $\A$-indexed  set of 
$l$-tuples from some $|\A|^+$-saturated structure $\M$
$$X = (\ov{c}_i \mid i \in \A) . $$
We want to find an $\A$-indexed indiscernible set of $l$-tuples from $\M$
$$Y = (\ov{e}_i \mid i \in \A)$$
such that $Y \vDash \emtp(X)$.

Let $g: \A \raw \B$ and $f: \B \raw \A$ witness that $\A$ is a semi-retraction of $\B$.  Define
$$X ' = (\ov{c}_{f(j)} \mid j \in \B) . $$
By assumption, there is a $\B$-indexed indiscernible set from $\M$
$$Y' = (\ov{d}_j \mid j \in \B)$$
such that $Y' \vDash \emtp(X')$ (recall that $|\A|=|\B|$ by Observation \ref{csb}, so $\M$ remains sufficiently-saturated).

Let $\ov{e}_i =\ov{d}_{g(i)}$.   It remains to show that $Y = (\ov{e}_i \mid i \in \A)$ is the desired set.

\ms

To see that $Y$ is an $\A$-indexed indiscernible set, fix $\ov{\imath}_1 \sim_\A \ov{\imath}_2$.  Since $g$ is qftp-respecting, $g(\ov{\imath}_1) \sim_\B g(\ov{\imath}_2)$.  By $\B$-indexed indiscernibility of $Y'$:
$$\ov{d}_{g(\ov{\imath}_1)} \equiv_{\M} \ov{d}_{g(\ov{\imath}_2)}$$
i.e.
$$\ov{e}_{\ov{\imath}_1} \equiv_{\M} \ov{e}_{\ov{\imath}_2} .$$

Now fix a complete quantifier-free $n$-type $\eta$, and an $n \cdot l$-ary formula $\varphi$ such that 
\begin{eqnarray}\label{1}
(\forall \ii) (\A \vDash \eta(\ov{\imath}) \Rightarrow \M \vDash \varphi(\ov{c}_{\ov{\imath}})) .
\end{eqnarray}
Also fix $\ov{s}$ so that
\begin{eqnarray}\label{2}
\A \vDash \eta(\ov{s}) .
\end{eqnarray}
To see that $Y \vDash \emtp(X)$, we wish to show that $\M \vDash \varphi(\ov{e}_{\ov{s}})$.

Since $f$ is qftp-respecting, there is an index set $S$ and there are some quantifier-free $n$-types $(\delta_k \mid k \in S)$ in $\B$ such that for any $n$-tuple $\ov{\jmath}$ from $\B$,
\begin{eqnarray}\label{3}
\A \vDash \eta(f(\ov{\jmath})) \Leftrightarrow \B \vDash \bigvee_{k \in S} \delta_k(\ov{\jmath}) ,
\end{eqnarray}
equivalently,
$$f^{-1}(\eta(\A)) = \bigcup_{k \in S} \delta_k(\B) .$$
Thus, via assumptions \eqref{1} and \eqref{3} we get that for all $k \in S$,
\begin{eqnarray}\label{4}
(\forall \ov{\jmath})(\B \vDash \delta_k(\ov{\jmath}) \Rightarrow \M  \vDash \varphi(\ov{c}_{f(\ov{\jmath})})) .
\end{eqnarray}
Since $Y' \vDash \emtp(X')$, by Proposition \ref{useful} we get that for all $k \in S$,
\begin{eqnarray}\label{5}
(\forall \ov{\jmath})(\B \vDash \delta_k(\ov{\jmath}) \Rightarrow \M  \vDash \varphi(\ov{d}_{\ov{\jmath}})) .
\end{eqnarray}

By condition \ref{red3} of Definition \ref{sr}:
\begin{eqnarray}\label{6}
\A \vDash \eta(\ov{s}) \Rightarrow \A \vDash \eta((f \circ g)(\ov{s})) .
\end{eqnarray}

Observe that by \eqref{3} and letting $\ov{\jmath} = g(\ov{s}) ,$
\begin{eqnarray}\label{7}
\A \vDash \eta((f \circ g)(\ov{s})) \Rightarrow \B \vDash \bigvee_{k \in S} \delta_k(g(\ov{s})) .
\end{eqnarray}
So we conclude by \eqref{2}, \eqref{6} and \eqref{7}:
\begin{eqnarray}\label{8}
\B \vDash \bigvee_{k \in S} \delta_k(g(\ov{s})) .
\end{eqnarray}
Apply this fact to \eqref{5} letting $\ov{\jmath} = g(\ov{s})$ to get 
\begin{eqnarray}\label{9}
\M \vDash  \varphi(\ov{d}_{g(\ov{s})}) 
\end{eqnarray}
i.e.
\begin{eqnarray}\label{10}
\M \vDash  \varphi(\ov{e}_{\ov{s}}) 
\end{eqnarray}
as desired.
\end{proof}

\begin{rmk} In the proof for Theorem \ref{thm1}, $Y = (\ov{d}_{g(i)} \mid i \in \A)$ is an $\A$-indexed indiscernible set because $(\ov{d}_j \mid j \in \B)$ is a $\B$-indexed indiscernible set and the map $g$ is qftp-respecting.  
It is only in verifying that $Y \vDash \emtp(X)$ that we use the map $f$ and condition \ref{red3} in Definition \ref{sr}.
\end{rmk}

\begin{cor}\label{transfer} Let $\A$ and $\B$ be locally finite ordered structures.  Suppose that $\A$ is a semi-retraction of $\B$ and $\B$ has RP.  Then $\A$ has RP.
\end{cor}

\begin{proof}  By Theorem \ref{thm1} and Theorem \ref{character}.
\end{proof}

\begin{cor}\label{treecor} If $\I_\str$ has RP, then $\I_\eq$ has RP.
\end{cor}

\begin{proof}   Define $\A$ to be the structure on the underlying set $\omega \times \mathbb{Q}$ with the same definition as $\I_\eq$ on $\omega \times \omega$ (thus, $\age(\A) = \age(\I_\eq)$ and each equivalence class in $\A$ is densely ordered by $\prec$).
Define $\B$ to be the structure on the underlying set $\leftexp{<\omega}{\mathbb{Q}}$ with the same definition as $\I_\str$ on $\leftexp{<\omega}{\mathbb{\omega}}$ (thus, $\age(\B) = \age(\I_\str)$ and the $\unlhd$-successors of any fixed node in $\B$ are densely ordered by $\lx$).
It remains to show that $\A$ is a semi-retraction of $\B$, in order to apply Corollary \ref{transfer}.

Our referee for \cite{kks14} kindly suggested that we deduce RP for $\A$ from RP for $\B$ by constructing a special embedding $g: \A \raw \B$ that is qftp-respecting (see \citep[Theorem 5.5]{kks14} for details).  
Given $i \in \omega$, by the $i$th level in $\B$, we mean all sequences in $\leftexp{<\omega}{\mathbb{Q}}$ of length $i$, and by the $i$th equivalence class in $\A$, we mean $\{(i,x) \mid x \in \mathbb{Q} \}$.

Let $\eta_i = {\underbrace{\la 0, \ldots, 0 \ra}_{2i}}$.  Let $g$ take the $i$th equivalence class in $\A$ into $\{\eta_i^\smf\la j \ra \mid j \in \mathbb{Q}_{>0}\}$ in a way that preserves the order.
Let $f: \B \raw \A$ be the map that takes the $i$th level in $\B$ into the $i$th equivalence class in $\A$ in a way that preserves the order.
$\A$ is a semi-retraction of $\B$ witnessed by $g$ and $f$.
\end{proof}

\begin{rmk} In Corollary \ref{treecor}, we have an example of $f, g$ witnessing that $\A$ is a semi-retraction of $\B$, such that $f \circ g$ is an embedding, but $g \circ f$ is not an embedding.  
\end{rmk}

\section{transfer by color-homogenizing maps}

We start with a technical definition.

\begin{dfn}[color-homogenizing maps]\label{ch} Fix ordered structures $\V$ and $\W$, and integers $m, k \geq 1$.  Given a finite substructure $\B \subseteq \V$, a $k$-coloring $c$ on increasing $m$-tuples from $\W$ and an 
increasing function $g: \B \raw \W$, we say that \textbf{$g$ is color-homogenizing for $c$ and $\B$} if for all increasing $m$-tuples $\ii, \jj$ from $\B$,
$$\ii \sim_\V \jj  \Rightarrow c(g(\ii)) =  c(g(\jj)) .$$
\end{dfn}

\begin{thm}\label{thm2} Let $\V$ and $\W$ be any ordered structures such that $\V$ is locally finite.  Suppose that there is an increasing function $f: \W \raw \V$ such that for any integers $m,k \geq 1$, any finite substructure $\B \subseteq \V$ and any $k$-coloring $c$ on increasing $m$-tuples from $\W$, there is a color-homogenizing map $g$ for $c$ and $\B$ such that $f \circ g: \B \raw \V$ is a $\sigma(\V)$-embedding.

Then $\V$ has RP.
\end{thm}

\begin{proof}
Fix finite substructures $\A$ and $\B$ of $\V$ and suppose that $||\A||=m$.  Let $c'$ be a $k$-coloring of ${\V \choose \A}$.  We may define a $k$-coloring $c''$ of all finite increasing $m$-tuples from $\V$ with the property that for any $\A' \in {\V \choose \A}$, we define $c''(\ov{a}') = c'(\A')$, where $\ov{a}'$ is the increasing enumeration of $\A'$.  It suffices to find a copy $\B'$ of $\B$ in $\V$ that is homogeneous for $c''$ on increasing copies of $\A$, i.e. such that for all increasing copies $\doh_1,\doh_2$ of $\A$ in $\B'$, $c''(\doh_1) = c''(\doh_2)$.

Define a $k$-coloring $c$ on increasing $m$-tuples $\so$ from $\W$ by $c(\so) = c''(f(\so))$.
By assumption, there is an increasing function $g: \B \raw \W$ that is color-homogenizing for $c$ and $\B$.

Let $\B' = (f\circ g)(\B)$.  By the assumption that $f\circ g$ is an embedding, $\B' \cong \B$.  To complete the argument that $\V$ has RP, it suffices to show that $\B'$ is homogeneous for $c''$ on increasing copies of $\A$.
Let $\doh_1,\doh_2$ be increasing copies of $\A$ in $\B'$, thus, $\doh_1 \sim_\V \doh_2$.  The embedding $f \circ g$ is order-preserving and surjective onto $\B'$.  Thus, we may define increasing tuples $\ii = {(f\circ g)}^{-1}(\doh_1)$ and $\jj ={(f\circ g)}^{-1}(\doh_2)$ from $\B$.  Since $\doh_1 \sim_\V \doh_2$ and $f\circ g$ is an embedding, we have that $\ii \sim_\V \jj$.  Since $\ii$ and $\jj$ are increasing tuples from $\B$ and $\ii \sim_\V \jj$, $c(g(\ii))=c(g(\jj))$, by the fact that $g$ is color-homogenizing for $c$ and $\B$.  By definition of $c$, $c''(f(g(\ii)))=c''(f(g(\jj)))$, i.e., $c''(\doh_1)=c''(\doh_2)$.
\end{proof}

\section{semi-direct product structures}

In this section, we focus on ordered relational structures.  In the following definition, it is convenient to assume that the relation symbol for order is common to the structures.  If this is not the case for certain desired input structures, we assume that we make it the case before applying the definition (as in Definition \ref{notn}).  The following operation is the ``disjoint sum'' operation on structures, \citep[p.101]{ho93} plus the requirement that $\prec$ be extended to a total order on the sum.



\begin{dfn}\label{Udef} Given a linear order $\Oo = (|\Oo|,\prec)$ 
and structures $(\M_i)_{i \in \Oo}$ on pairwise-disjoint domains $|\M_i|$ in relational signatures $L_i$, each linearly ordered by $\prec \in L_i$, define $\U=\U_{i \in \Oo}(\M_i)$ to be the structure on $\bigcup_{i \in \Oo} |\M_i|$ in the signature 
\begin{eqnarray}
\displaystyle \sigma(\U) = \{P_\alpha \mid \alpha \in \Oo\} \cup \bigcup_{i \in \Oo} L_i   \nonumber
\end{eqnarray}
\noindent for new unary predicates $P_\alpha$, where the symbols are interpreted as follows.

\begin{enumerate}[label=(\roman*)]
\item 
For each $n$-ary relation symbol $R_\ell$ that is not $\prec$ or any of the $P_\alpha$, 
$R_\ell^\U = \bigcup_{i \in \Oo} X_i$, where we define

$$X_i = 
\begin{cases}
R_\ell^{\M_i} &, \textrm{~if~} R_\ell \in L_i\\
\emptyset &, \textrm{~if~} R_\ell \notin L_i  
\end{cases}
$$
\item $a \prec^{\U}b$ if and only if there exist $i, j \in \Oo$ such that $a \in \M_i, b \in \M_j$ and either $i \prec^{\Oo} j$ or else $i=j$ and $a \prec^{\M_i} b$,
\item $P^\U_\alpha = |\M_\alpha|$.
\end{enumerate}
\end{dfn}

\begin{obs} \label{unique}
\begin{enumerate}
\item $\U_{i \in \Oo}(\M_i)$ in Definition \ref{Udef} is linearly ordered by $\prec$. 
\item Given a finite substructure $\A \subseteq \U_{i \in \Oo}(\M_i)$, there is a unique integer $s \geq 1$, a unique finite sequence $t_0 \prec \ldots \prec t_{s-1}$ from $\Oo$, and unique substructures $\A_i \subseteq \M_{t_i}$, for all $i<s$, such that $\A = \bigcup_{i<s} \A_i$.
\end{enumerate}
\end{obs}


\begin{prop}\label{AP} If $\age(\M_i)$ has AP, for all $i \in \Oo$, then $\age(\U)$ has AP.
\end{prop}

\begin{proof} Since the language is relational, we may allow empty structures in order to simplify our argument.

Fix an amalgamation problem in $\age(\U)$, $f_1 : \A \raw \B, f_2: \A \raw \C$.  We may assume that all $\A, \B, \C \subseteq \U$.  Since the language is relational, structures $\A, \B, \C$ are finite.

By Observation \ref{unique}(2), there exist unique increasing tuples $\ov{s}=(s_i)_{i<n}, \ov{t}=(t_j)_{j<m}$ from $\Oo$ and substructures $\B(i) \subseteq \M_{s_i}, \C(j) \subseteq \M_{t_j}$ such that $\B = \bigcup_{i<n} \B(i)$, $\C = \bigcup_{j<m} \C(j)$.  Let $Y=\ran \ov{s}, Z= \ran \ov{t}$, and $X= Y \cup Z$.  For $k \in X$, define $\B_k =\B(i)$, if $k=s_i$, and otherwise $\B_k = \emptyset$.  Likewise, define $\C_k =\C(j)$, if $k=t_j$, and otherwise $\C_k = \emptyset$.  Define $\A_k = \A \cap \M_k$, for all $k \in X$.

For each $k \in X$, we define a structure $\D_k \in \age(\U)$ and embeddings $g_1^k : \B_k \raw \D_k, g_2^k: \C_k \raw \D_k$ as follows.  

If $k \in Y \setminus Z$, then $k=s_i$ for some unique $i<n$, and we let $g_1^k: \B_k \raw \B_k$ be the identity function, $\D_k =\B_k$ and $g_2^k = \emptyset$.

If $k \in Z \setminus Y$, then $k=t_j$ for some unique $j<m$, and we let $g_2^k: \C_k \raw \C_k$ be the identity function, $\D_k =\C_k$ and $g_1^k = \emptyset$.

If $k \in Y \cap Z$, then $k=s_i=t_j$ for some unique $i<n, j<m$. 
This is the only case where $\A_k$ could possibly be nonempty, since embeddings must preserve the predicates $P_\alpha$, for all $\alpha \in \Oo$.
%
The restrictions of $f_1, f_2$, respectively, $f_1^k : \A_k \raw \B_k$, $f_2^k : \A_k \raw \C_k$, form an amalgamation problem in $\age(\M_k)$.  By assumption, there is a solution to the amalgamation problem $g_1^k, g_2^k, \D_k$ such that $g_1^k \circ f_1^k = g_2^k \circ f_2^k$.

Define $g_1= \bigcup_{i<n} g_1^i$ and $g_2= \bigcup_{j<m} g_2^j$.
For each $k \in X$,  $g_1^k$ and $g_2^k$ are $L_k$-embeddings on substructures of $\M_k$, and so $g_1$ and $g_2$ preserve $\prec$ and the $P_\alpha$, for all $\alpha \in \Oo$, and thus are $\sigma(\U)$-embeddings.
Let $\D = \bigcup_{k \in X} \D_k$.
It is not hard to check that $\D$, $g_1$, $g_2$ is a solution to the amalgamation problem such that $g_1 \circ f_1 = g_2 \circ f_2$.

\end{proof}

Here we restate the product Ramsey theorem for classes.  By the notation ${(\B_i)_{i<s} \choose (\A_i)_{i<s}}$ we mean all 
sequences $(\A_i')_{i<s}$ such that $\A_i' \subseteq \B_i$ and $\A_i' \cong \A_i$, for every $i< s$.

\begin{thm}[{\citep[Theorem 2]{sok211}}]\label{sokprod} Fix integers $r, s \geq 1$ and let $(\K_i)_{i<s}$ be a sequence of classes of finite structures with RP.  Fix $(\B_i)_{i<s}, (\A_i)_{i<s}$ such that $\B_i, \A_i \in \K_i$, for all $i<s$.  There exist $\C_i \in \K_i$ for all $i<s$ such that for any coloring $p: {(\C_i)_{i<s} \choose (\A_i)_{i<s}} \raw r$, there exists a sequence $(\B_i')_{i<s}$, with $\B_i' \cong \B_i$ for all $i<s$  and some $l \in r$ such that $p$ restricted to ${(\B_i')_{i<s} \choose (\A_i)_{i<s}}$ is the constant function $l$.
\end{thm}

\begin{cor}\label{prod} If $\M_i$ has RP, for all $i \in \Oo$, then $\U_{i \in \Oo}(\M_i)$ has RP.  
\end{cor}

\begin{proof} Let $\U=\U_{i \in \Oo}(\M_i)$. It suffices to show that $\U$ has the $\A$-Ramsey property for all finite $\A \subseteq \U$.  Fix two structures $\A \subseteq \B$.  By Observation \ref{unique}(2), there is a unique decomposition $\B = \bigcup_{i<s} \B_i$ where each $\B_i \subseteq \M_{t_i}$, for some $t_i \in \Oo$, for all $i<s$.   We may write $\A = \bigcup_{i<s} \A_i$ where $\A_i \subseteq \M_{t_i}$, if we allow some of the $\A_i$ to be empty.  Now apply Theorem \ref{sokprod} to $(\K_{t_i})_{i<s}$, $(\A_i)_{i<s}$ and $(\B_i)_{i<s}$.
\end{proof}


We generalize Definition \ref{Udef}.

\begin{dfn}[semi-direct product structures]\label{sdps} Given a structure $\N$ in a relational signature $L_2$ linearly ordered by $\prec \in L_2$ and, for some relational signature $L_1$ such that $L_1 \cap L_2 = \{\prec\}$,
$L_1$-structures $(\M_i)_{i \in \N}$ on pairwise-disjoint domains $|\M_i|$, each linearly ordered by $\prec$,
define
$\I=\I_{i \in \N}(\M_i)$ to be the structure on $\bigcup_{i \in \N} |\M_i|$ in the signature
$$\sigma(\I) = L_1 \cup L_2 \cup \{E\}$$
\noindent for a new binary relation $E$, where the symbols are interpreted as follows.

\begin{enumerate}[label=(\roman*)]
\item For each $n$-ary relation symbol $S_\ell \in L_1 \setminus \{\prec\}$, define
$$S_\ell^\I = \bigcup_{i \in \N} S_\ell^{\M_i} .$$
\item For each $n$-ary relation symbol $R_l \in L_2 \setminus \{\prec\}$, define 
	$R_l^{\I}(a_0,\ldots,a_{n-1})$ to hold if and only if there exist (possibly non-distinct) elements $\{t_0,\ldots, t_{n-1}\}$ from $\N$ such that $a_i \in \M_{t_i}$ for all $i<n$ and $\N \vDash R_l(t_0,\ldots,t_{n-1})$.
\item Define $E$ to be an equivalence relation whose equivalence classes are exactly the $|\M_i|$, i.e., $E^\I(a,b) \Leftrightarrow (\exists i \in \N)(a, b \in \M_i)$.
\item Define $a\prec^{\I}b$ if and only if there exist $i, j \in \N$ such that $a \in \M_i, b \in \M_j$ and either $i \prec^\N j$ or else both $i=j$ and $a \prec^{\M_i} b$.
\end{enumerate}
\end{dfn}

\begin{obs}\label{sigma}
\begin{enumerate}
\item $\I=\I_{i \in \N}(\M_i)$ in Definition \ref{sdps} is linearly ordered by $\prec$.
\item Let $\I^-$ be the $L_2$-reduct of $\I$.  For $R_l \in L_2 \setminus \{\prec\}$, the definable sets $R_l^\I$ are $E^\I$-invariant.  Moreover, $\prec^\I$ is $E^\I$-invariant on pairs $(a,b) \notin E^\I$. Thus, we may form the quotient structure $(\I^-)/E$, which is an $L_2$-structure.
\item Define a map $\sigma : {(\I^-)/E} \raw \N$ by $[a]/E \mapsto t$, if $a \in \M_t$.  It is clear that this map is well-defined, bijective, and preserves the interpretations of all symbols in $L_2$.  Thus $\sigma$ is an $L_2$-isomorphism.
\end{enumerate}
\end{obs}

\begin{dfn}\label{forthm} Let $\N, L_1, L_2, \{\M_i\}_{i \in \N}, \I=\I_{i \in \N}(\M_i)$ be as in Definition \ref{sdps}, and let $\N'$ be the $\{\prec\}$-reduct of $\N$.  Let $f_E : \I \raw (\I^-)/E$ be the map that takes $a$ to its equivalence class $[a]/E$.

\begin{enumerate}
\item For a substructure $\A \subseteq \I_{i \in \N}(\M_i)$, define $\gr(\A)$ to be the $L_2$-substructure of $\N$ identified with $f_E(\A)$ in Observation \ref{sigma}(3).

We call $\gr(\A)$ the \emph{underlying graph} of $\A$.  
\item For a tuple $\nub$ from $\I_{i \in \N}(\M_i)$, by $\gr(\nub)$ we mean $\gr(\textrm{ran~}{\nub})$.   
\item Given two substructures $\A, \B \subseteq \I$ and a $\{\prec,E\}$-embedding $f: \A \raw \B$, define $\ov{f} : \gr(\A) \raw \gr(\B)$ to be the $\{\prec\}$-embedding given by $\ov{f}( \sigma([a]/E) ) = \sigma([f(a)]/E)$, where $\sigma$ is the function named in Observation \ref{sigma}(3).
\item For a substructure $\C \subseteq \I_{i \in \N}(\M_i)$, by $\C^{\red}$ we mean the $(L_1 \cup \{E\})$-reduct of $\C$.  For $\C\subseteq \U_{i \in \N'}(\M_i)$, by $\C^{\red}$ we mean the $(L_1 \cup \{E\})$-reduct of $\C$ as it is naturally interpreted, meaning the sets $P_\alpha^\C$, for $\alpha \in \N'$, are defined to be exactly the $E^{\C^{\red}}$-equivalence classes.
\end{enumerate}
\end{dfn}

\begin{obs} $|\gr(\A)| = \{t \in |\N| \mid |\A| \cap |\M_t| \neq \emptyset \}$.  We use this notation to point out similarities with the partite construction in \cite{rone89}.  There are also similarities with the argument in \citep[Proposition 1]{sok111}.
\end{obs}

\begin{prop}\label{red}  Let $\N, L_1, L_2, \{\M_i\}_{i \in \N}, \I=\I_{i \in \N}(\M_i)$ be as in Definition \ref{sdps}.  Let $\N'$ be the $\{\prec\}$-reduct of $\N$, and let $\U = \U_{i \in \N'}(\M_i)$.

Fix $\ov{a}, \ov{b} \in \leftexp{m}{|\I|} (= \leftexp{m}{|\U|})$ such that
\begin{eqnarray}\label{11}
\qftp^{\I^{\red}}(\ov{a})= \qftp^{\I^{\red}}(\ov{b}) \label{gr1}
\end{eqnarray}
and let $f$ be the map defined by $f(a_i)= b_i$ for all $i<m$.
Then
\begin{enumerate}[label=(\roman*)]
\item\label{help1} $\ov{f}: \gr(\ov{a}) \raw \gr(\ov{b})$ is the identity function  if and only if $\qftp^{\U}(\ov{a})= \qftp^{\U}(\ov{b})$, and
\item\label{help2} $\ov{f}: \gr(\ov{a}) \raw \gr(\ov{b})$ is an isomorphism  if and only if $\qftp^{\I}(\ov{a})= \qftp^{\I}(\ov{b})$.
\end{enumerate}
\end{prop}

\begin{proof}  Recall that $P_t^\I = |\M_t|$, for all $t \in \N'$.  By  assumption  \eqref{11}, there exist $q \in \omega$ and sequences $(s_i)_{i<q}, (t_i)_{i<q}$ such that for all $j<m$, for all $i<q$: $a_j \in |\M_{s_i}| \Leftrightarrow b_j \in |\M_{t_i}|$.  In other words, $a_j \in P_{s_i}^\U \Leftrightarrow b_j \in P_{t_i}^\U$.  

To see \ref{help1}, note that under  assumption  \eqref{11}, $\qftp^{\U}(\ov{a})= \qftp^{\U}(\ov{b})$ holds if and only if $s_i = t_i$, for all $i<q$.
However, $\ov{f}: \gr(\ov{a}) \raw \gr(\ov{b})$ is the identity function if and only $E^\I(a_j,b_j)$ for all $j<m$, which holds if and only if $s_i=t_i$, for all $i<q$.

To see \ref{help2}, note that under assumption  \eqref{11},  $f$ is a $\sigma(\I)$-isomorphism if and only if $f$ is an $L_2$-embedding.  It remains to show that $f$ is an $L_2$-embedding if and only if $\ov{f}$ is an $L_2$-embedding.

Fix an $n$-ary relation symbol $R_l \in L_2 \setminus \{\prec\}$.
By the definition of $\I$, $R_l^\I(a_{i_0},\ldots,a_{i_{n-1}}) \Leftrightarrow R_l^\N(s_{i_0},\ldots,s_{i_{n-1}})$, for all $i_0, \ldots, i_{n-1} < m$.  Similarly, $R_l^\I(b_{i_0},\ldots,b_{i_{n-1}}) \Leftrightarrow R_l^\N(t_{i_0},\ldots,t_{i_{n-1}})$, for all $i_0, \ldots, i_{n-1} < m$. 
Moreover, $\ov{f}$ is an $L_2$-embedding if and only if $R_l^\N(s_{i_0},\ldots,s_{i_{n-1}}) \Leftrightarrow R_l^\N(t_{i_0},\ldots,t_{i_{n-1}})$, for all $n$-ary $R_l \in L_2$, $i_0, \ldots, i_{n-1} <m$ and $n \in \omega$.  Thus, $f$ is an $L_2$-embedding if and only if $\ov{f}$ is an $L_2$-embedding.
\end{proof}

We give a slight restatement of RP that we use in the proof of Theorem \ref{thm3}.

\begin{dfn} Given a finite substructure $\B$ of some structure $\V$,
define a $k$-coloring of $ {\V \choose \age(\B)}$ to be a $k$-coloring of $ \bigcup_{\A \in \age(\B)} {\V \; \choose \A}$
\end{dfn}

\begin{prop}\label{thesis} If $\V$ has RP, then for any finite substructure $\B \subseteq \V$, for any integer $k \geq1$ and $k$-coloring $c$ of ${\V \choose \age(\B)}$, there is a copy $\B'$ of  $\B$ in $\V$ such that $\B'$ is homogeneous for $c$ on copies of $\A$, for all $\A \in \age(\B)$.  As a generalization of our usual convention we say that $\B'$ is \textbf{homogeneous for $c$}.
\end{prop}

\begin{proof} This is well-known (see Claim 4.16 in \cite{sc12}) and can also be argued for using $\V$-indexed indiscernible sets. We repeat the argument here.  List $\age(\B) = \{\D_1,\ldots,\D_{m}\}$.  Let $c$ be a $k$-coloring of  $ {\V \choose \age(\B)}$ .  Define structures $(\W_i \mid 1 \leq i \leq m+1)$ such that $\W_1 =\B$ and $\W_{n} \raw (\W_{n-1})^{\D_{n-1}}_k$, for all $n$ such that $2 \leq n \leq m+1$.  Now define $\V_1=\W_{m+1}$.  Having defined $\V_{n-1}$, define $\V_n$ to be a copy of $\W_{m-(n-2)}$ in $\V_{n-1}$ homogeneous for $c$ on copies of $\D_{m-(n-2)}$, for all $n$ such that $2 \leq n \leq m+1$.  Thus $\V_1 \supseteq \V_2 \supseteq \cdots \supseteq \V_{m+1}$ and ultimately $\V_{m+1}$ is a copy of $\W_1 = \B$ homogeneous for $c$ on copies of $\D_m, \D_{m-1}, \ldots, \D_1$.
\end{proof}

\begin{thm}\label{thm3} Let $\N, L_1, L_2, \{\M_i\}_{i \in \N}, \I=\I_{i \in \N}(\M_i)$ be as in Definition \ref{sdps}.  Let $\N'$ be the $\{\prec\}$-reduct of $\N$ and let $\U = \U_{i \in \N'}(\M_i)$.  Assume that there is an age, $\K$, that has RP and such that $\age(\M_i) = \K$, for all $i \in \N$.

If $\N$ has RP, then $\I$ has RP.
\end{thm}

\begin{proof} 
The structures $\I$, $\U$ share their underlying set which we call $X$.  
Define $f:|\U| \raw |\I|$ to be the identity map on underlying sets.  By the interpretation of $\prec$ on both structures, $f$ is an increasing function. 

To simplify notation, we adopt the convention that given any finite substructures
$\C_1 \subseteq \I$ and $\C_2 \subseteq \U$ and a bijection $p: \C_1 \raw \C_2$,
we say that $p$ is a $\sigma(\I)$-isomorphism if $f \circ p$ is truly a $\sigma(\I)$-embedding.

Fix some integers $m, k \geq 1$, a finite substructure $\B \subseteq \I$ (which we may assume to be of cardinality at least $m$),  and a $k$-coloring $c$ on increasing $m$-tuples from $\U$.  Representatives of isomorphism types of cardinality-$m$ substructures of $\B$ may be listed as:
$$\A_0,\ldots,\A_{t-1}$$
for some $t \in \omega$.

Let $\Hh = \gr(\B)$, and let $d = k^t$.  
By the assumption that $\N$ has RP and Proposition \ref{thesis}, there is some finite substructure $\N_0 \subseteq \N$ such that
\begin{eqnarray}\label{N}
\N_0 \raw (\Hh)^{\age(\Hh)}_d
\end{eqnarray}
For every $\Hh_0 \subseteq \N_0$ such that $\Hh_0 \cong_{L_2} \Hh$, there exists a finite substructure $\B_0 \subseteq \U$ and an $(L_1 \cup \{E\})$-isomorphism $\tau : \B_0^{\red} \raw \B^\red$ such that $\gr(\B_0)=\Hh_0$ and $\ov{\tau}: \Hh_0 \raw \Hh$ is an $L_2$-isomorphism.  This is because $\age(\M_i) = \K$, for every $i \in \N$, and $\sigma(\U)$ is relational.



Simply by being an age, $\age(\U)$ has JEP, and so there exists a finite structure $\Bb \subseteq \U$ that embeds each of the finitely many $\B_0$ described in the paragraph immediately above.  We may additionally assume that this structure $\Bb$ has the property that $\gr(\Bb) = \N_0$.

\

We may define an $m$-coloring $c^*$ of ${\U \choose \age(\Bb)}$ with the property that
for any $\A' \in {\U \choose \age(\Bb)}$ such that $||\A'||=m$, $c^*(\A') = c(\ov{a}')$, where $\ov{a}'$ is the increasing enumeration of $\A'$.

By Corollary \ref{prod}, $\U_{i \in \N'}(\M_i)$ has RP, so by Proposition \ref{thesis}, there is a copy $\Bb^*$ of $\Bb$ in $\U$ such that $\Bb^*$ is homogeneous for $c^*$.  By Proposition \ref{red}(i),
$\gr(\Bb^*)=\N_0=\gr(\Bb)$.

Making use of Observation \ref{equiv}, we define a $d$-coloring $c': {\N_0 \choose \age(\Hh)} \raw \leftexp{t}{k}$ that maps into a set of size $d$ as follows: for an $L_2$-structure $\J \subseteq \N_0$ such that $\J \in \age(\Hh)$, define $c'(\J) = (k_0,\ldots,k_{t-1})$
where we define, for any $i<t$,

$$k_i=
\begin{cases}
0 		&, \textrm{~if~} \gr(\A_i) \not\cong_{L_2} \J \\
c^*(\D) 	&, \textrm{~if~} \gr(\A_i) \cong_{L_2} \J \textrm{~and there exists~} \D \subseteq \Bb^* \textrm{~such that~} {\D}^{\red} \cong {\A_i}^{\red} \\
& \hspace{3.75in} \textrm{~and~} \gr(\D) = \J
\end{cases}
$$

To see that $c'$ is well-defined, consider structures $\D, \hat{\D} \subseteq \Bb^*$ such that $\gr(\D) = \gr(\hat{\D})$ and there exists an isomorphism
$h: \D^\red \raw {\hat{\D}}^\red$.
The map $\ov{h}: \gr(\D) \raw \gr(\hat{\D})$ is a $\{\prec\}$-embedding and $\gr(\D)$ is finite, so $\ov{h}$ is the identity map.  By Proposition \ref{red}(i), $h: \D \raw \hat{\D}$ is a $\sigma(\U)$-isomorphism, and so by homogeneity of $\Bb^*$ for $c^*$, $c^*(\D) = c^*(\hat{\D})$.

By line \eqref{N}, there is a copy $\Hh^*$ of $\Hh$ in $\N_0$ such that $\Hh^*$ is homogeneous for $c'$.

Now we refer to the construction of $\Bb$ and the fact that $\Bb^* \cong \Bb$.  Since $\Hh^* \subseteq \N_0$ has the property that $\Hh^* \cong_{L_2} \Hh$, there exists a finite substructure $\B^* \subseteq \Bb^*$ and an $(L_1 \cup \{E\})$-isomorphism $\tau : {\B^*}^{\red} \raw \B^{\red}$ such that $\gr(\B^*)=\Hh^*$ and $\ov{\tau}: \Hh^* \raw \Hh$ is an $L_2$-isomorphism.

We apply Theorem \ref{thm2} where $g$ is defined to be $\tau^{-1}$.
By Proposition \ref{red}(ii), since $\ov{\tau}: \Hh^* \raw \Hh$ is an $L_2$-isomorphism, $\tau : \B^* \raw \B$ is  a $\sigma(\I)$-isomorphism.  Thus,
\begin{eqnarray}\label{L2}
g: \B \raw \B^* \textrm{~is a~} \sigma(\I)\textrm{-isomorphism,}
\end{eqnarray}
which by our convention means that $f \circ g : \B \raw \I$ is a $\sigma(\I)$-embedding.

To see that $g$ is a color-homogenizing map for $c$ and $\B$, fix any increasing $m$-tuples $\ov{a}, \ov{a}'$ from $\B$ such that  $\ov{a} \sim_\I \ov{a}'$.  It remains to show that $c(g(\ov{a})) = c(g(\ov{a}'))$.

Since $\ov{a}$ and $\ov{a}'$ are $m$-tuples from $\B$, there is some $s<t$ such that $\ov{a}$ and $\ov{a}'$ are increasing copies of $\A_s$.
Since $\ov{a}$ and $\ov{a}'$ are increasing tuples and $g$ is a $\{\prec\}$-embedding, $g(\ov{a}), g(\ov{a}')$ are also increasing tuples.  Define substructures $\D, \D' \subseteq \U$ such that $g(\ov{a})$ is the increasing enumeration of $\D$ and $g(\ov{a}')$ is the increasing enumeration of $\D'$.  
Since $\ov{a}$ and $\ov{a}'$ are from $\B$, $g(\ov{a})$ and $g(\ov{a}')$ are from $\B^* \subseteq \Bb^*$.  Since $\Bb^* \cong_{\sigma(\U)} \Bb$, both $\D$ and $\D'$ are structures of cardinality $m$ in ${\U \choose \age(\Bb)}$, and so $c^*(\D) = c(g(\ov{a}))$ and $c^*(\D') = c(g(\ov{a}'))$, by definition.  Thus, it remains to show that $c^*(\D) = c^*(\D')$.

Let $\J = \gr(\D)$ and $\J'=\gr(\D')$.  By line \eqref{L2}, $g(\ov{a}) \sim_\I \ov{a} \sim_\I \ov{a}' \sim_\I g(\ov{a}')$.
This yields that $\D^\red \cong {\A_s}^\red \cong {\D'}^\red$ and $\J = \gr(\D) \cong \gr(\A_s) \cong \gr(\D') = \J'$, by Proposition \ref{red}(ii).
Since $\Hh^*$ is homogeneous for $c'$ and $\J, \J' \subseteq \gr(\B^*)=\Hh^*$, $c'(\J) = c'(\J')$.  In particular, $c^*(\D) = (c'(\J))_s = (c'(\J'))_s = c^*(\D')$, as desired.
\end{proof}


\section{applications}

Theorem \ref{thm3} yields interesting examples, some of which are familiar.  First we introduce a definition that is well-defined up to bi-definability of structures.

\begin{dfn}\label{notn}  Given an ordered relational structure $\M$ and an order $\Oo=(|\Oo|,\prec)$, by $\U_{i \in \Oo}(\M)$, we mean $\U_{i \in \Oo}(\M_i)$ where $(\M_i)_{i \in \Oo}$ is a sequence of isomorphic copies of $\M$ on pairwise-disjoint domains and the symbol $\prec$ has been substituted for the symbol for order in $\sigma(\M)$.

Given two ordered relational structures $\N$ and $\M$, $\I_{i \in \N}(\M)$ is defined to be $\I_{i \in \N}(\M_i)$ where $(\M_i)_{i \in \N}$ is a sequence of isomorphic copies of $\M$ on pairwise-disjoint domains, the symbol $\prec$ has been substituted for the symbols for order in both $\sigma(\N)$ and $\sigma(\M)$, and by additional substitution of symbols, we have made $\sigma(\N) \cap \sigma(\M) = \{\prec\}$.
\end{dfn}


\begin{cor}\label{rpCor} Let $\Oo = (|\Oo|,\prec)$ be a linear order and $\N$ a random ordered graph.  If $\M$ is an ordered relational structure with RP then
\begin{enumerate}
\item $\I_{i \in \Oo}(\M)$ has RP
\item $\I_{i \in \N}(\M)$ has RP
\end{enumerate}
\end{cor}

\begin{proof}  By Theorem \ref{rgRP}, $\N$ has RP.  That $\Oo$ has RP follows from Ramsey's theorem for finite sequences.  Thus $\I_{i \in \Oo}(\M)$ and $\I_{i \in \N}(\M)$ have RP, by Theorem \ref{thm3}.
\end{proof}

\begin{rmk}  Corollary \ref{rpCor}(1) is obtained by Leeb using the notation Ord($\mathcal{\C}$) in \cite{le73} (see \cite{graro74} for a discussion).  
\end{rmk}

\begin{cor}\label{examples}
\begin{enumerate}
\item Let $\Oo= (\omega,<)$.  Then $\I_{i \in \Oo}( \Oo)$ has RP.  
\item If $\N$ is the random ordered graph in the signature $\{R, \prec\}$, then $\I_{i \in \N}(\N)$ has RP.
\end{enumerate}
\end{cor}

\begin{proof} Note that $\I_{i \in \Oo}( \Oo )$ is isomorphic to the structure $\I_\eq$ defined in Definition \ref{trees} (assuming $\{E, \prec\}$ is the common signature) which structure is known to have RP. Alternatively we could use Ramsey's theorem for finite sequences, which guarantees that $\Oo$ has RP, so that we may apply Corollary \ref{rpCor}.  For the second claim, we use Theorem \ref{rgRP} to conclude that $\N$ has RP, and thus $\I_{i \in \N}(\N)$ has RP by Corollary \ref{rpCor}(2).
\end{proof}

\begin{rmk}\label{xhelp}  Let $\R = \I_{i \in \N}(\N)$ where $\N$ is the random ordered graph in the signature $\{R,\prec\}$.  By Definition \ref{notn}, $\sigma(\R) = \{R_1, R_2, E, \prec\}$ where we may assume $R_1$ is substituted for the edge relation symbol in $L_1$ and $R_2$ is substituted for the edge relation symbol in $L_2$, where $L_1$ and $L_2$ are as in Definition \ref{sdps}.  We may define an interpretation of $R$ on $\R$ such that $R^\R = R_1^\R \cup R_2^\R$.  Note that for all $a, b \in \R$, $\R \vDash R_1(a,b) \Leftrightarrow \R \vDash R(a,b) \wedge E(a,b)$ and $\R \vDash R_2(a,b) \Leftrightarrow \R \vDash R(a,b) \wedge  \neg E(a,b)$, the latter since $R^\N$ is irreflexive.  Thus, in this case, $\R$ is interdefinable with an $\{R, E, \prec\}$-structure on the same underlying set.
\end{rmk}

\begin{dfn} A theory $T$ has the \emph{independence property (IP)} if there is a
partitioned formula $\varphi(\ov{x}; \ov{y})$ in the language of the theory with the following
property: for every $n \in \omega$, there exist parameters $(\ov{a}_s \mid s<n)$ and $(\ov{b}_t \mid t<2^n)$ from some model of the theory, such that
$$\varphi(\ov{b}_t; \ov{a}_s) \Leftrightarrow s \in w_t$$
where $(w_t \mid t < 2^n)$ enumerates the subsets of $n$.

If a theory fails to have the independence property, we say that the theory \emph{has NIP}, or \emph{is an NIP theory}.
\end{dfn}

We end with a characterization of NIP theories using the example in Corollary \ref{examples}(2).

\begin{cor}\label{rg} Fix a random ordered graph $\N$ in the signature $L=\{R,\prec\}$.

Let  $\R=\R_{i \in \N}(\N)$. 

Let $\W$ be the $\{E,\prec\}$-reduct of $\R$. 

A theory $T$ has NIP if and only if any $\R$-indexed indiscernible sequence in a 
model $\M$ of $T$ is a $\W$-indexed indiscernible sequence.
\end{cor}

\begin{proof}  First observe that, by Remark \ref{xhelp}, $\sigma(\R) = \{R_1, R_2, E, \prec\}$ where we may assume $R_1$ is substituted for the edge relation symbol in $L_1$ and $R_2$ is substituted for the edge relation symbol in $L_2$.

Let $\N' = \N \uphp \{<\}$.
We observe that $\W = \I_{i \in \N'}(\N')$.

The right-to-left direction follows the argument in \citep[Lemma 5.2]{sc12} closely, so we merely sketch it here.  By Corollary \ref{examples}(2), $\R$ has RP, so by Theorem \ref{character}, $\R$-indexed indiscernible sequences have the modeling property.  
If $T$ has IP, then there are parameters in a model of $T$ witnessing this and an $\R$-indexed indiscernible sequence $(\ov{a}_i \mid i \in \R)$ from some $\aleph_1$-saturated elementary extension $\M \vDash T$ locally-based on these parameters such that $(\ov{a}_i \mid i \in \W)$ is not a $\W$-indexed indiscernible sequence.

For the left-to-right direction, we can adapt the argument in \citep[Lemma 5.4]{sc12}.  Assume that there is an $\R$-indexed indiscernible sequence $(\ov{a}_i \mid i \in \R)$ from some model $\M \vDash T$ that is not a $\W$-indexed indiscernible sequence.  It is convenient to think of $i \mapsto \ov{a}_i$ as a map, so that we may refer to $\ov{a}_i$ as the \emph{image} of $i$.

Since $(\ov{a}_i \mid i \in \R)$ is not a $\W$-indexed indiscernible sequence, there exist $n \in \omega$ and $n$-tuples $\ii, \jj$ from $\R$ such that
$$\ii \sim_\W \jj ,$$
but
$$\ov{a}_{\ii} \not\equiv_{\M} \ov{a}_{\jj} .$$
Thus, there is some formula $\varphi \in L(\M)$ such that
$$\M \vDash \varphi(\ov{a}_{\ii}) \textrm{~and~} \M \vDash \neg \varphi(\ov{a}_{\jj}) ,$$
and so by $\R$-indexed indiscernibility of the sequence,
$${\ii} \not\sim_{\R} {\jj} .$$

Since $R^\N$ is symmetric with no loops, complete quantifier-free $n$-types in $\R$ are of the form
\begin{eqnarray}\nonumber
p(x_0,\ldots,x_{n-1}) \cup \{R_1(x_i,x_j)\mid (i,j) \in Y \wedge i<j \} \cup \{\neg R_1(x_i,x_j) \mid (i,j) \in (n \times n) \setminus Y  \wedge i<j \}  \\\nonumber
\cup \{R_2(x_i,x_j)\mid (i,j) \in Z \wedge i<j \} \cup \{\neg R_2(x_i,x_j) \mid (i,j) \in (n \times n) \setminus Z  \wedge i<j \}
\end{eqnarray}
where $p$ is a complete quantifier-free $n$-type in $\W$, and $Y, Z \subseteq n \times n$.  
%

For any $s \in \{1,2\}$, let $R_{s}(x_i,x_j)^0$ denote $R_{s}(x_i,x_j)$ and $R_{s}(x_i,x_j)^1$ denote $\neg R_{s}(x_i,x_j)$.
Given a complete quantifier-free $n$-type $q'$ and some $s \in \{1,2\}$, define $t_s(q')=0$ if $R_{s}(x_i,x_j) \in q'$, and otherwise $t_s(q')=1$.
For any $s \in \{1,2\}$, for every pair $i<j<n$ and any complete quantifier-free $n$-type $q'$, define $\tau^s_{(i,j)}(q') = (q' \setminus \{R_{s}(x_i,x_j)^{t_s(q')}\}) \cup \{R_{s}(x_i,x_j)^{1-t_s(q')}\}$
Let $\Gamma = \{\tau^s_{(i,j)} \mid i<j<n \wedge s \in \{1,2\}\}$.

Assume that $q_1$ is the quantifier-free type of $\ii$ in $\R$ and $q_2$ is the quantifier-free type of $\jj$ in $\R$.  
Since $\ii \sim_{\W} \jj$, the complete quantifier-free types $q_1, q_2$ agree on their restriction $p^*(x_0,\ldots,x_{n-1})$ to the signature $\{E,\prec\}$.  Thus, since ${\ii} \not\sim_{\R} {\jj}$, there is some integer $m \geq 1$ and some finite sequence $(\tau_j)_{j < m}$ from $\Gamma$ such that $(\tau_{m-1} \circ \cdots \circ \tau_1 \circ \tau_0)(q_1) = q_2$.  We may additionally assume that for each $j <m$, $(\tau_{j-1} \circ \cdots \circ \tau_1 \circ \tau_0)(q_1)$ is a complete quantifier-free $n$-type in $\R$.  One way to do this would be get rid of any $R_2$-edges specified by $q_1$, and then add in any $R_2$-edges specified by $q_2$, one-by-one.  Since the $R_2$-edges are specified by $q_2$, they only hold of $E$-inequivalent pairs, and thus at every step we have a type consistent with the theory of $\R$.  After this point, $R_1$-edges may be flipped as necessary within the $E$-classes.

There is a least $j_0 < m$ such that all tuples $\ov{s}$ satisfying $(\tau_{j_0-1} \circ \cdots \circ \tau_1 \circ \tau_0)(q_1)$ in $\R$ yield an image $\ov{a}_{\ov{s}}$ satisfying $\neg \varphi$ in $\M$.  Replacing $q_1$ by $(\tau_{j_0-2} \circ \cdots \circ \tau_1 \circ \tau_0)(q_1)$ and $q_2$ by $(\tau_{j_0-1} \circ \cdots \circ \tau_1 \circ \tau_0)(q_1)$,
we may assume that there is some $(s,t) \in n \times n$, some $l \in \{1,2\}$ and some quantifier free type $q^*$ such that $q^* = q_1 \cap q_2$, which we shall call the \emph{common quantifier free type}, and such that $q_1 = q^* \cup \{R_l(x_s,x_t)\}$ and $q_2 = q^* \cup \{\neg R_l(x_s,x_t)\}$ (the assignment of $R_l$ and $\neg R_l$ is also without loss of generality, by switching $\varphi$ with $\neg \varphi$).  For convenience, we write $q^* = q^*(x_s,x_t,\ov{u})$, where we define $\ov{u} = (x_i)_{i \in (n \setminus \{s,t\})}$.


In the first case, assume that $l=0$.  Then $E(x_s,x_t) \in p^*$ and $q_1$ and $q_2$ disagree on $R_1(x_s,x_t)$.  Note that $\N$ is a \fr~limit, in both of its roles in $\R=\I_{i \in \N}(\N)$.  Thus, there is some $(n-2)$-tuple $\ov{c}$ from $\R$ such that we may realize arbitrary finite bipartite graphs $(A,B)$ with edge relation $R_1$ as induced subgraphs of the class $[i_s]/E(=[i_t]/E)$ with the property that for all $a \in A$, for all $b \in B$, $\R \vDash q^*(a,b,\ov{c})$.  This allows the images of the tuples $(a,b,\ov{c})$ in $\M$ to satisfy IP using $\varphi$ partitioned as $\varphi(x_s;x_t,\ov{u})$.

In the second case, assume that $l=1$.  Then $\neg E(x_s,x_t) \in p^*$ and $q_1$ and $q_2$ disagree on $R_2(x_s,x_t)$.  We may realize arbitrary finite bipartite graphs on pairs $([a]/E,[b]/E)$ in the quotient structure (mentioned in Observation \ref{sigma}(2)) with edge relation $R_2$ and with all the required $R_2$-relations to some fixed $(n-2)$-tuple $[\ov{c}]/E$, as dictated by the common quantifier-free type.  Then all $R_1$-configurations are easily found within the classes $[a]/E$, $[b]/E$, $[c_i]/E$ to match the common quantifier free type, and so we may complete this argument as in the first case.
\end{proof}

\section*{Acknowledgements}\label{last}

Thanks go to the referee for comments and suggestions that improved this paper.
The author thanks Dana Barto\v{s}ov\'{a} for pointing out that the \texttt{qfi} condition is not required for the proof of Theorem \ref{character}.  Thanks go to Miodrag Soki\'{c} for the useful references and explanation of the history around Theorem \ref{sokprod}, and to Dugald Macpherson and John Baldwin for help with terminology.  Thanks go to Thomas Scanlon for comments that improved an earlier draft of this paper.  The author thanks Charles Steinhorn for a discussion of these ideas in their early stages.

\bibliographystyle{plainnat}
\bibliography{prerefab2020}

\end{document}